\documentclass[
preprint
]{elsarticle}

\usepackage{geometry}
\usepackage{hyperref} 
\usepackage{graphicx}
\usepackage{amsthm}
\usepackage{amsmath}
\usepackage{amssymb}
\usepackage{dsfont}
\usepackage{rotating}
\usepackage{lineno}
\usepackage{subfigmat}
\usepackage{float}
\usepackage[active]{srcltx}

\def\Z{\mathbb{Z}} \def\N{\mathbb{N}} 
\def\cm{\mathcal{M}} 
\def\fg{\mathfrak{g}} 
\def\ds{\displaystyle}

\def\Z{\mathds{Z}} 
\def\N{\mathds{N}}

\newtheorem{corollary}[section]{Corollary}
\newtheorem{lemma}{Lemma}
\newtheorem{proposition}{Proposition}

\newtheorem{example}{Example}

\journal{Arxiv} 

\begin{document}

\begin{frontmatter}
\title{A note on quadrangular embedding of Abelian Cayley Graphs\tnoteref{label1}}
\tnotetext[label1]{Partially supported by  FAPESP (Grants 
2007/56052-8, 2007/00514-3 and 2011/01096-6) and CNPq (Grants 09561/2009-4)}

\author[FCA]{{J.~E.~Strapasson}\corref{cor1}}
\ead{joao.strapasson@fca.unicamp.br}

\author[UNICAMP]{{S.~I.~R. Costa}}
\ead{sueli@ime.unicamp.br}

\author[UFPr]{{M.~Muniz}}
\ead{marcelo@mat.ufpr.br}

\cortext[cor1]{Corresponding Author}
\address[UNICAMP]{Institute of Mathematics, University of Campinas, 13083-970, Campinas, S\~ao Paulo, Brazil.}
\address[FCA]{School of Applied Sciences, University of Campinas, 13484-350,
	Limeira, S\~ao Paulo, Brazil.}
\address[UFPR]{Mathematics Department, Federal University of Paran\'a,
Curitiba, Paran\'a, Brazil}


\begin{abstract}
\noindent The genus graphs have been studied by many authors, but just a few results concerning in special cases: Planar, Toroidal, Complete, Bipartite and Cartesian Product of Bipartite. We present here a derive general lower bound for the genus of a abelian Cayley graph and construct a family of circulant graphs which reach this bound. 
\end{abstract}


\end{frontmatter}

\section{Introduction}

The genus of a graph, defined as the minimum genus of a $2$-dimensional
surface on which this graph can be embedded without crossings (\cite%
{GroTuc,Trud}), is well known as being an important measure of the graph
complexity and it is related to other invariants.

A circulant graph, $C_{n}(a_{1},\dots,a_{k}),$ is an homogeneous graph which
can be represented (with crossings) by $n$ vertices on a circle, with two
vertices being connected if only if there is jump of $a_i$ vertices from one
to the other (Figure \ref{cir_13_1_6}). A circulant graph is particulary case of abelian Cayley graph. Different aspects of circulant
graphs have been studied lately, either theoretically or through their
applications in telecommunication networks and distributed computation \cite{LinYangLuHao,Matthew,Lisko,Muzy,Felix,Heub03}.

Concerning specifically to the genus of circulant graphs few results are
known up to now. We quote \cite{Boesch} for a small class of toroidal (genus
one) circulant graphs, \cite{Heub03} which establish a complete
classification of planar circulant graphs, \cite{grande} which establish a complete
classification of toroidal circulant graphs, and the cases where the circulant
graph is either complete or a bipartite complete graph (\cite%
{BeHA65,Hara,Ring65a,Ring65b,RiYo68}).

In \cite{Costa07} we show how any circulant graph can be viewed as a quotient of lattices and obtain as consequences that: i) for $k=2$, any circulant graph must be either genus one or zero (planar graph) and ii) for $k=3$, there are circulant graphs of arbitrarily high genus.

We present here a 
derive a general lower bound for the genus of abelian Cayley graph $C_{n}(a_{1},\dots,a_{k})$ as $\ds \frac{(k-2)\,n+4}{4}$, (Proposition \ref{limitante}), and construct a family of abelian Cayley graphs which reach this bound (Corollary \ref{circuquad}).

This note is organized as follows. In Section 2 we introduce concepts and previous results concerning circulant graphs, abelian Cayley graphs and genus. 
In Section 3 we derive a lower bound for the genus of an $n$-circulant graphs of order $2\,k$ (Proposition \ref{limitante}) and construct families of graphs reaching this bound for arbitrarily $k$ (Corollary \ref{circuquad}).

\section{Notation and Previous Results}

In this section we recall concepts and results used in this paper concerning acirculant graphs and fix the notations.

Let $G=(\{e=g_1, \dots g_n\},+)$ be a finite abelian group.  Given a subset $S=\{a_1,\dots, a_k\}$ of $G$, the associated \emph{Cayley graph} $(G,S)$ is an undirected graph whose vertices are the elements of $G$, and where two vertices $g_i$ and $g_j$ are connected if and only if $g_i-g_j=\pm a_l$ for some $a_l \in S$. We remark that $(G,S)$ is connected if and only if $S$  generates $G$ as a group, and that this graph is $2k$-regular if $a_i+a_i\neq 0, \forall i=1,2,\dots k,$ and $2k-l$-regular otherwise, where $l$ is a number of $a_i$ such that $a_i+a_i = 0$.

A \emph{circulant graph $C_{n}(a_{1},\dots,a_{k})$} with $n$ vertices $v_{0},\dots v_{n-1}$ and jumps $a_{1},\dots,a_{k}$, $0<a_{j}\leqslant\lfloor n/2\rfloor$, $a_{i}\neq a_{j}$, is an undirected graph such that each vertex $v_{j},0\leqslant j\leqslant n-1$, is adjacent to all the vertices $v_{j\pm a_{i}\mod n}$, for $1\leqslant i\leqslant k$. A circulant graph is homogeneous: any vertex has the same order (number of
incident edges), with is $2\,k$ except when $a_{j}=\frac{n}{2}$ for some $j$, when the order is $2\,k-1$, a circulant graphs is particulary case of abelian Cayley graph $(G=\Z_n,S=\{a_1, \dots ,a_k\})$.

The $n$-cyclic graph and the complete graph of $n$ vertices are examples of circulant graphs denoted by $C_{n}(1)$ and $C_{n}(1,\dots,\lfloor n/2\rfloor)$, respectively. Figure \ref{cir_13_1_6} shows on the left the standard picture of the circulant graph $C_{13}(1,6)$. 

\begin{figure}[H]
\hfill \includegraphics[scale=0.33]{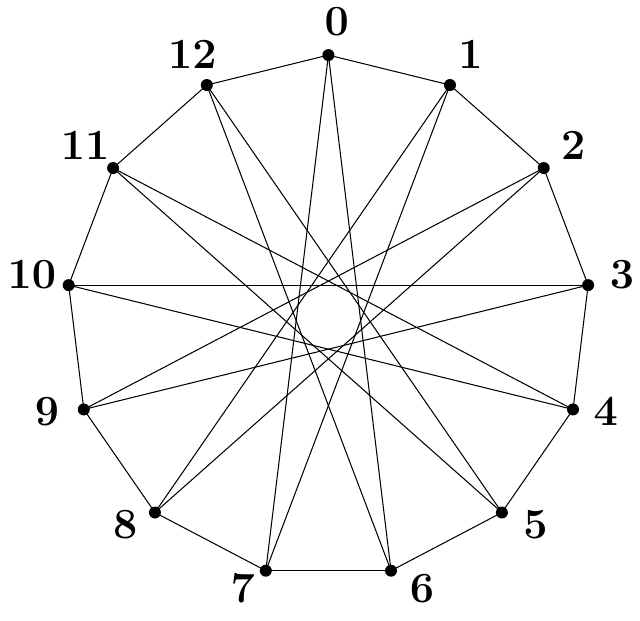}
\hfill \includegraphics[scale=0.33]{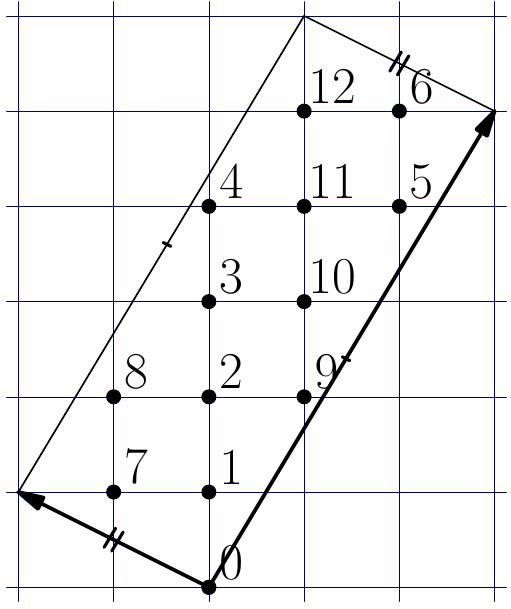} \hfill \,
\caption{The circulant graph $C_{13}(1,6)$ represented in the standard form (left) and on a 2-dimensional flat torus (right). }
\label{cir_13_1_6}
\end{figure}

In what follows we write $(a_{1},\dots,a_{k})=(\tilde{a}_{1},\dots,\tilde{a}_{k})\mod n$ to indicate that for each $i$, there is $j$ such that $a_{i}=\pm\tilde{a}_{j}\mod n$. Two circulant graphs, $C_{n}(a_{1},\dots,a_{k})$ and $C_{n}(\tilde{a}_{1},\dots,\tilde{a}_{k})$ are said to satisfy the \emph{\'{A}d\'{a}m's relation} if there is $r$, with $\gcd(r,n)=1$, such that

\begin{equation}  \label{adam}
(a_{1},\dots, a_{k})=r\,(\tilde a_{1},\dots, \tilde a_{k}) \mod n
\end{equation}

An important result concerning circulant graphs isomorphisms is that circulant graphs satisfying the \'{A}d\'{a}m's relation are isomorphic (\cite{Adam}). The reciprocal of this statement was also conjectured by \'{A}d\'{a}m. It is false for general circulant graphs but it is true in special cases such as $k=2$ or $n=p$ or $n=pq$ ($p$ and $q$ prime) (see \cite{Lisko,Alspach}). In this paper we will not distinguish between isomorphic graphs.

Without loss of generality we will always consider $a_1<\dots <a_k\leqslant
n/2$ for a circulant graph $C_{n}(a_{1},\dots,a_{k})$.

A circulant graph $C_{n}(a_{1}, \dots, a_{k})$ is connected if, and only if,  $\gcd(a_{1}, \dots, a_{k},n)\allowbreak=1$ (\cite{Boesch}). \textit{In this paper we just consider connected circulant graphs}.

The genus of a graph is defined as the minimum genus, $\fg$, of a $2$-dimensional orientable compact surface $\cm_{\fg}$ on which this graph can be embedded without crossings (\cite{GroTuc,Trud}). This number, besides being a measure of the graph complexity, is related to other invariants.

A graph $E$ is an \textit{expansion} of $H$ if it is constructed from $H$ by possibly adding new vertices on the edges of $H$. Finally, if there is an expantion $E $ of $H$ which is a subgraph of $G$ we say $G$ is \textit{supergraph} of $H$. From this definition follows that 

$$\mbox{\textit{if $G$ is a supergraph of $H$, genus($G$) $\geqslant$
genus($H$).}}   \label{supergraph}$$

When a connected graph $G$ is embedded on a surface, $\cm_\fg$, of minimum genus $\fg$ it splits the surface in regions called \textit{faces}, each one homeomorphic to an open disc surrounded by the graph edges, giving rise to a tessellation on this surface. Denoting the number of faces, edges and vertices by $f$, $e$, and $v$ respectively, those numbers must satisfy the well known Euler's second relation:

\begin{equation}
v+f-e=2-2\,\fg   \label{euler}
\end{equation}

We quote next other known relations those numbers must satisfy (\cite%
{GroTuc,Trud}):

If $G$ is a graph of genus $\fg$ with $v\geqslant l$ such that any face in $\cm_\fg$ has at
least $l$ sides in its boundary,

\begin{equation}
l\,f\leqslant2\,e \text{ and } \fg\geqslant\dfrac{l-2}{2\,l}\,e-\dfrac{1}{2}\,(v-2).   \label{liminfgenus}
\end{equation}

In the above expressions we have equalities if, only if, all the faces have $%
l$ sides.

An upper bound for the genus of a connected graph of $n$ vertices is given
by the genus of the \textit{complete graph}, $C_{n}(1,\dots,\lfloor
n/2\rfloor)$, which is $\left\lceil \frac{(n-3)(n-4)}{12}\right\rceil $.
Combining the lower bound above with a minimum of three edges
for each face, we can write the following inequality, for $n\geqslant3$:

\begin{equation}
\left\lceil \frac{1}{6}e-\frac{1}{2}(n-2)\right\rceil \leqslant
\fg\leqslant\left\lceil \frac{(n-3)(n-4)}{12}\right\rceil ,
\label{limitantetrudeau}
\end{equation}
where $\left\lceil x \right\rceil$ is the ceiling (smallest integer which is greater or equal to) of $x$.

For a circulant graph $C_{n}(a_{1}, \dots, a_{k})$, $a_{1}< a_{2}< \dots<
a_{k}$ we can replace $e$ by $e=n\, k$ when $a_k<\frac n 2$, or $e=
n\,(2k-1)/2$ when $a_k=\frac n 2$. We can then rewrite the lower bound in
last expression as $\left\lceil \frac{n}{6}(k-3)+1 \right\rceil $ or $%
\left\lceil \frac{n}{6}(k-4)+1 \right\rceil $, respectively.

\subsection{Previous results on genus of circulant graphs and abelian Cayley graphs}

\begin{itemize}
\item \emph{\textbf{Theorem}(Ringel, Beineke and Harary, 1965 \cite{Ring65c,BeHA65})
The genus of the $n$-cube graph $Q_n$ is $1+2^{n-3}(n-4)$.}

\item \emph{\textbf{Theorem}(Ringel, 1965 \cite{Ring65a,Ring65b})
The genus of the \textit{complete bipartite graph} $K_{m,n}$ is $\left\lceil \frac{(m-2)(n-2)}{4}\right\rceil $.}
Since $K_{n,n}$ is the circulant graph $C_{2\,n}(1,3,\dots,2\left\lceil \frac{n-1}{2}\right\rceil -1)$, the genus of this one-parameter family is $\left\lceil \frac{(n-2)^{2}}{4}\right\rceil$.

\item \emph{\textbf{Theorem}(White, 1970 \cite{Whit70})
Let $G=C_{m_1}\square C_{m_2} \cdots \square C_{m_r}$, where $C_{m_i}$ is even cycle, $r>1$ and $m_i>3$ for all $i$. Then the genus of $G$ is $1 + v(G)(r - 2)/4$}

\item \emph{\textbf{Theorem}(Pisanski, 1980 \cite{Pisa80})
Let $G$ and $H$ be connected $r$-regular bipartites graphs. Then Cartesian product $G\square H$ of $G$ and $H$ has genus $1+pm(r-2)/4$ where $p$ and $m$ are the number of vertices of $G$ and $H$, respectively.}

\item \emph{\textbf{Theorem}(Heuberger, 2003 \cite{Heub03}\label{genus0}) 
A planar circulant graph is either the graph $C_{n}(1)$, or $C_{n}(a_{1}, a_{2})$, where i) $a_{2}=\pm2\,a_{1} \mod n$ and $2\vert n$, ii) $a_{2}=n/2$, and $2\vert a_{2}$.}

\item For $k=2,$ and general $(a_{1},a_{2}),$ we have shown that circulant
graphs $C_{n}(a_{1},a_{2})$ are very far for from reaching the upper bound
for the genus given in \eqref{limitantetrudeau}, as it was shown in \cite{Costa07}:

\begin{proposition}[\cite{Costa07} \label{genus1}]
Any circulant graph $C_{n}(a_{1},a_{2})$, $a_{1}<a_{2}\leqslant n/2$, has genus one, except for the cases of planar graphs: i) $a_{2}=\pm2\,a_{1} \mod n,$ and $2|n$, ii) $a_{2}=n/2$, and $2|a_{2}$.
\end{proposition}

\item For $k=3$ and $n\neq2\,a_{3}$ we can assert that the genus of $%
C_{n}(a_{1},a_{2},a_{3})$ satisfies:

\begin{equation}
1\leqslant \fg\leqslant\left\lceil \frac{(n-3)(n-4)}{12}\right\rceil
\end{equation}

The genus of the complete graph $C_{7}(1,2,3)$ achieves the minimum value
one \eqref{limitantetrudeau}. However, in opposition to the case $k=2$, the
genus of a circulant graph $C_{n}(a_{1},a_{2},a_{3})$ can be arbitrarily
high:

\begin{proposition}[\cite{Costa07}]\label{genushigh} 
There are circulant graphs $C_{n}(a_{1},a_{2},a_{3})$ of arbitrarily high genus. A family of such graphs is given by: $n=(2\,m+1)\,(2\,m+2)\,(2\,m+3),\ m\geqslant2;$ $a_{1}=(2\,m+2)\,(2\,m+3),$ $a_{2}=(2\,m+1)(2\,m+2)\,(m+1),$ $a_{3}=(2\,m+2)\,(2\,m+3)\,(m+1),$ with the correspondent genus satisfying 

\begin{equation}
\fg\geqslant2\,m\,(m+1)^{2}+1.
\end{equation}
\end{proposition}
\end{itemize}

In the next section we deal with the more general class of abelian Cayley graphs and establish a lower bound for their genus.

\section{Quadrangular embedding of abelian Cayley graphs}
\label{quadri} 
In this section we consider Cayley graphs of abelian groups, a more general class of graphs of which circulant graphs form a very particular subclass, the Cayley graphs of cyclic groups. Nevertheless, an important feature of circulant graphs, their embeddings in $k$-dimensional tori, is  shared by the whole class of Cayley graphs of abelian groups.  We will see in the following that $k$ is associated to the number of elements of the generating set of the edges of the Cayley graph. We will determine a subclass of these graphs that has quadrangular embeddings, and hence a subclass where we know the genus of each graph.


It is known that graphs that have $3$-cycles may have embeddings with triangular faces, and some easy calculations establish a lower bound for the genus. In general we can also establish a lower bound that depends on the girth $l$ of the graph. If $(G,S)$ is a Cayley graph and there are no solutions of $a_{h}=\pm (a_{i}\pm a_{j})$ for $h,i,j \in \{1,2, \dots k\}$ (not necessarily distinct), the girth is always $4$ (a typical $4$-cycle is $0,a_i,a_i+a_j,a_j,0$), which implies at least four edges for each face.

If the graph is $2 k$-regular, then

\begin{equation*}
\fg\geqslant\frac{l-2}{2\,l}\,a-\frac{v-2}{2}=\frac{2}{8}\,n\,k-\frac{n-2}{2}=%
\frac{n\,k-2\,n+4}{4}.
\end{equation*}

Hence, we get the following lemma, which establishes a lower bound for circulant graphs with no triangular faces.

\begin{lemma} \label{limitante} 
The genus, $\fg$ of the circulant graph $C_{n}(a_{1},\dots,a_{k}),$ such that $a_{i}\neq a_{j}+a_{l},\ \forall\,i,j,l\leqslant k$ and $n\neq2\,a_{i},\ \forall i$ satisfy:

\begin{equation*}
\fg\geqslant\frac{n\,k-2\,n+4}{4}.
\end{equation*}
\end{lemma}

In what follows the (additive) subgroup of $G=\Z_{n_1} \times \Z_{n_2} \times \cdots \times  \Z_{n_l}$ generated by $a_1, \dots , a_k \in G$ is denoted by 
$\langle a_1, \dots , a_k \rangle$, and let $G_s = \langle a_1,a_2, \ldots,
a_{s}\rangle \triangleleft G$. Define $L_s$ for $s=1, \ldots, k$ by $L_1= o(a_1)$ is the order of $a_1$ in $G$ and $L_s = o(a_s +G_{s-1})=[G_s:G_{s-1}]$ if $1< s \leqslant k$, is the order of the class $a_s +G_{s-1}$ in $G/G_{s-1}$.

Under the above conditions we can assert that $x\in G$ can be expressed uniquely as a linear combination, $x=m_1 a_1+ m_2 a_2 + \dots m_k a_k$, where $0\leqslant m_i<L_i$. This fact is stated next a lemma. 

\begin{lemma}\label{decomp}
Given $x \in G$ and $G_s  \triangleleft G, 1 \leqslant s<k$ then there exist unique $m_i \in \N$ and $R_{s,x}$ such that 
\begin{equation*}
x = m_1\,a_1 + \cdots + m_s \,a_s +R_{s,x} \text{ and } R_{s,x}=R(m_{s+1},\dots ,m_{k}) = m_{s+1} a_{s+1} + \cdots m_k a_k
\end{equation*}
with $0 \leqslant m_i < L_i$ for all $i$. 
\end{lemma}


Through Lemma \ref{decomp} we can show that not only circulant graph \cite{Costa07} but any Cayley graph of an abelian group can be embedded in a $k$-dimensional torus. The construction of such embedding, for $k=2$, is illustrade in Figure \ref{cayleytoro}.

We consider the mapping:

\begin{eqnarray}
\begin{matrix}
\varphi: & \Z^k & \longrightarrow & G\\
 & (x_1, \dots ,x_k) & \longmapsto & x_1 a_1+\cdots x_k a_k.
\end{matrix}
\end{eqnarray}

Therefore $\dfrac{\Z^k}{\ker \varphi}\simeq G$ and $\ker \varphi$ is lattice. The Cayley graph associated to $G$ , as a quotient of lattices, is then naturally embedded in flat torus which a polytope generated by basis of this lattice, with the parallel faces identified.

To proceed in a uniform way we can use the standard Hermite basis for $\ker \varphi$, as it done for circulant graphs in \cite{Heub03}.

We remark that Hermite basis of $\ker \varphi$, $\{U_1, \dots , U_k\}$, is given as columns of a upper triangular matrix $(b_{i,j})_{k\times k}$, where $b_{i,i}=L_i$ and $0\leqslant b_{i,j}<L_i$. 

In figure \ref{cayleytoro}, we consider the Cayley graph of $G=\Z_2\times \Z_8$ and $a_1=(1,2)$ and $a_2=(0,1)$, therefore $o(a_1)=4$ and $o(a_2+\langle a_1\rangle)=4$ and Hermite basis is $\{(4,0),(2,4)\}$. 

\begin{figure}[H]
\begin{subfigmatrix}{2}
\subfigure[Flat torus]{\includegraphics[height=3.5cm]{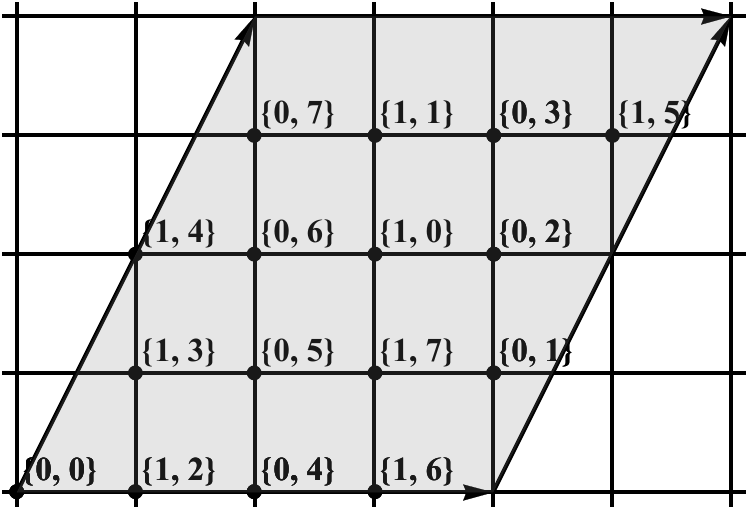}} 
\subfigure[Torus]{\includegraphics[height=3.5cm]{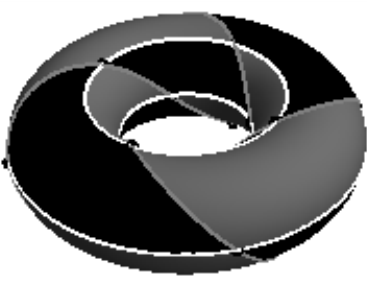}}
\end{subfigmatrix}
\caption{$2$-embedding of Cayley graph $(\Z_2\times \Z_8,\{(1,2),(0,1)\})$.}
\label{cayleytoro}
\end{figure}

%

We will construct the embedding of cayley graphs $(G,S)$ by induction on $k=\#S$. Note that $(G,S-\{a_k\})$ may be disconnected: it is
well known that this graph has $d$ components, where $L_k =[G=G_k:G_{k-1}]=o(a_k+G_{k-1})$, and that each component is isomorphic to Cayley graph $(G_{k-1},S-\{a_k\})$. Since $x$ and $y$ are linked by a path if and only if $x -y = m_1\,a_1 + \cdots + m_{k-1}\,a_{k-1}$ in $G$, it follows that $x$ and $y$ are in the same component if and only if $x \equiv y \in G_{k-1}$. Hence, each $0 \leqslant j_s < L_k$ determines a component and the numbers.

\begin{equation*}
R_{k,m_k} + m_1\,a_1 + \cdots + m_{k-1}\,a_{k-1} 
\end{equation*}
with $0 \leqslant m_k < L_k$ and $R_{k,m_k}$ is a fixed element this component, describe all the vertices of the component of $(G,S-{a_k})$ associated to $m_k$.

\begin{proposition}
\label{cayley} 
Let $(G,\{a_{1},\dots ,a_{k}\})$, where $n=\#G=2 l$, and $L_i=2 l_i, 1\leqslant i \leqslant k$, and $l_1>1$ and $a_{i}\neq \pm (a_j \pm a_h),\ 1\leqslant i,j,k < k$. Hence, the genus of $G$ is $\dfrac{n\, k-2\, n+4}{4}$.
\end{proposition}

\textit{Proof.} The proof will be done by induction on $k$. For $k=2$ it is trivial (vide Figure \ref{cayleytoro}). Assume that the result holds for $k-1$. The graph $(G,\{a_{1}, \dots, a_{k-1}\}$ is a disconnected Cayley graph
with an even number of connected components ($L_k=2l_k$), $H_{m_k}=R_{k,m_k}+G_{k-1}$, $0\leqslant m_k<2l_k$, where $2l_k = [G:G_{k-1}]$. Each component $H_{m_k}$ can be embedded on a
surface $S_{m_k}$ giving to a rise tessellation where every face has 4 edges. As in
figure \ref{exempinv}, we reverse the orientation of the components that
contain the odd multiples of $a_{k}$. We wish to add tubes that,
topologically, are prisms with squared bases.

\begin{figure}[H]
	\centering 		\includegraphics[height=4cm]{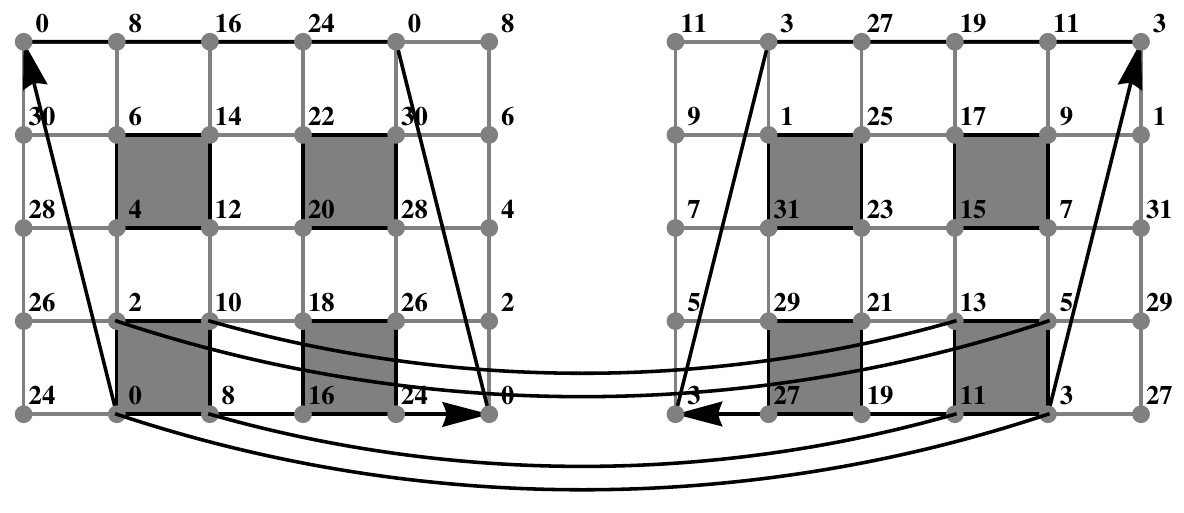} 
\caption{The construction of the $C_{32}(8,2,3)$ embedding. Tubes are added on the two connected components of $C_{32}(8,2)$ considered with reversal orientation.}%
\label{exempinv}%
\end{figure}

Let $x \in H_{m_k}$, then $x = m_k a_k + m_1 a_1 + \cdots + m_{k-1}a_{k-1}$
where $0 \leqslant m_k < 2l_k$ and $0 \leqslant m_i < 2 l_i$ as in Lemma \ref{decomp}.
Given $0<j< k-1$, we can also express

\begin{eqnarray}
x &=& m_k a_k + (2p_j + \delta_j)\,a_j + (2p_{k-1} + \delta_{k-1})\,a_{k-1}+ g \text{ and}\\
x &=& m_k a_k + (2q_j + \delta_j)\,a_j + ((2q_{k-1}+1) - \delta_{k-1})\,a_{k-1}+ g \ (q_j>0)
\end{eqnarray}
with $g=\ds \sum_{\substack{ i = 1, \ldots, k-2  \\ i \neq j}}m_i a_i$ and $\delta_j, \delta_{k-1} \in \{0,1\}$. 

Three possibilities should be considered:

i) Each vertex of $H_{m_k}$, $m_k$ even, is a vertex of a square determined by $\{P,P + a_{k-1},P + a_{k-1} + a_{j},P +a_{j}\}$, where $P$ is of the form

\begin{equation}  \label{formak}
P=m_k a_k + 2p_j \,a_j + 2p_{k-1}\,a_{k-1}+ g,
\end{equation}
where $0\leqslant 2\,p_{j} < L_{j}$, $2\,p_{k-1} < L_{k-1}$. Just as Figure \ref{exempinv}, each such square is then connected to the square $\{P+a_k,P +a_k+ a_{k-1},P +a_k+ a_{k-1} + a_{j},P +a_k+a_{j}\}$ (which lies on $H_{m_k+1}$) by a prism which contains the edges $[P,P+a_k], [P+a_{k-1},P+a_k+ a_{k-1}], [P+ a_{k-1} + a_{j},P+a_k+ a_{k-1} + a_{j}]$ and $[P+ a_{j},P+a_k+ a_{j}]$; and then we cut out both squares. Doing this for every $j \in \{1,2, \ldots, k-2\}$ we construct a surface where each edge of the form $[x,x+a_k]$ is embedded without crossings.

ii) Each vertex of $H_{m_k}$, $m_k$ odd and $m_k\neq 2l_k-1$, is a vertex of a square determined by $\{Q,Q + a_{k-1},Q + a_{k-1} + a_{j},Q +a_{j}\}$, where $Q$ is of the form 
\begin{equation}  \label{formak2} 
Q=m_k a_k + 2q_j \,a_j + (2q_{k-1}+1)\,a_{k-1}+ g, 
\end{equation}
where $2\leqslant 2q_{j} \leqslant L_{j}$, $2\,q_{k-1} < L_{k-1}$. We the same reasoning as above can be applied by replacing $P$ by $Q$.

iii) Each vertex, $x$ of $H_{2l_k-1}$, $x+a_k \in H_0$. This case requires special care, since we need to choose a face in $H_{2l_k-1}$ and another in $H_0$, once that some faces have been excluded. This choice depends on how $2 l_k a_k$ is described in $G_k$, since $2 l_k a_k=\tilde m_j a_j + \tilde m_{k-1} a_{k-1} +\tilde g$, $\tilde g =\ds \sum_{\substack{ i = 1, \ldots, k-2  \\ i \neq j}}\tilde m_i a_i$, choose $\tilde Q$ such that $m_j + \tilde m_j$, $m_{k-1}+\tilde m_{k-1}$ and $m_j$ are not both even, and $m_{k-1}$ is not odd. We then repeat the procedure of item (i), replacing $P$ by $\tilde Q$.

Therefore, under restrictions considered in this proposition we always can connect the excluded squares by prisms and construct a surface which is tessellated by $(G,S)$, and each face of this tessellation is a square. This last remark shows that this surface has the required genus, and this
concludes the proof. 
\qed

\begin{corollary}
\label{circuquad} 
Let $G=C_{n}(a_{1},\dots ,a_{k})$, where $n=2^r l$, $a_i=2^{r_i} l_i, i=1,\dots k-1$, where $l,l_i, a_k$ odd, $0<r_{i+1}<r_{i}<r$ and $a_{i+1}\neq \pm 2\,a_i,\ 1\leqslant i <k$. Hence, the genus of $G$ is $\dfrac{n\, k-2\, n+4}{4}$.
\end{corollary}

The next example shows that there are more circulant graphs than the ones considered in Proposition \ref{cayley} which also can be embedded giving rise to a quadrilateral tessellation.

\begin{example}
For the graph $C_{32}(8,2,3,7)$, if we consider $C_{32}(8,2,3)$ as in last proposition, we note that just half the faces of the tessellation of $C_{32}(8,2)$ are excluded to add tubes. We can also exclude the other faces adding tubes to support the edges $\pm a_4$. Hence this is an embedding generating quadrilateral faces and since there are no cycles of size 3, the expression $\dfrac{n\,k-2\,n+4}{4}$ for the genus still holds. 
\end{example}

Figure \ref{genus32} shows all the circulant graphs of 32 vertices for which the genus can be given by Proposition \ref{genus0} (Heuberger), \ref{genus1} and Corollary \ref{circuquad}

\begin{figure}[H]
\centering\begin{tabular}{@{\extracolsep{\fill}}||c||c|c|c|c|c||}
\cline{1-6}
 $k$ & $a_{1}\in$ & $a_{2} \in$ & $a_{3} \in$ & $a_{4} \in$ & $\fg$ \\ \cline{1-6}
 $1$ & $I$ &  &  &  & $0$ \\
 $2$ & $I$ & $2\,(I-\{\pm a_{1}\})$ &  &  & $1$ \\
 $2$ & $I$ & $4\,I$ & & & $1$ \\
 $2$ & $I$ & $8\,I$ & - & - & $1$ \\
 $3$ & $I$ & $2\,(I-\{\pm a_{1}\})$ & $4\,(I-\{\pm2\,a_{2}\})$ & - & $9$ \\
 $3$ & $I$ & $2\,(I-\{\pm a_{1}\})$ & $8\,I$ & - & $9$ \\
 $3$ & $I$ & $4\,I$ & $8\,(I-\{\pm2\,a_{2}\})$ & - & $9$ \\
 $4$ & $I$ & $2\,(I-\{\pm a_{1}\})$ & $4\,(I-\{\pm2\,a_{2}\})$ & $8\,(I-\{\pm 2\,a_{3}\})$ & $17$ \\ \cline{1-6}
\end{tabular}
\caption{All circulant graphs of 32 vertices satisfying Propositions \ref{genus0} (Heuberger), \newline \ref{genus1}  and \ref{circuquad} ($I=\{\pm1,\pm3, \dots, \pm15\}$).}%
\label{genus32}%
\end{figure}

We note that some graphs satisfying the hypotheses of the last proposition belong to the class of graphs with given genus. For some of those graphs we could have used the results of White and Pisanski (see \cite{Whit70,Pisa80}) to determine their genus namely. The particular class of Cartesian product of bipartite graphs which satisfy the proposition hypothesis. However, many of the graphs considered in the last proposition are not Cartesian product of bipartite graphs. 

For example the Cayley graphs $G_1=(\Z_2\times \Z_8, \{(1,2),(0,1)\})$ and $G_2=C_{16}(1,4)$.

We can assert the $G_1$ is not a product of de bipartite graphs, since as $G_1=G \square H $, a regular graph, to be such a product, each graph factor needed to be also a regular graph. There are few possibilities to be considered here: $\#G=2$ or $4$ and $\#H=8$ or $4$, respectively. That is, either i) $G=P_2$ and $H$ is a $3$-regular bipartite complete (hypothesis of White and Pisanski results), what is not possible since $H$ is bipartite it should be $4$-regular. or ii) $G=H=K_{2,2}$ what again cannot be, since $K_{2,2}\square K_{2,2}=(\Z_4\times \Z_4,\{(1,0),(0,1)\})$, which has spectrum $\{8,6,6,6,6,4,4,4,4,4,4,2,2,2,2,0\}$ and $G_1$ has spectrum $\{8,6,6,4+\sqrt{2},4+\sqrt{2},4+\sqrt{2},4+\sqrt{2}, 4,4,4-\sqrt{2},4-\sqrt{2},4-\sqrt{2},4-\sqrt{2},2,2,0\}$.

$G_2$ also is not a product of bipartite graphs since it has odd size cycles (ex: $0,1,2,3,4,0$) what not occur for bipartite graphs.

%

\end{document}